\documentclass[12pt]{amsart}
\usepackage{amsmath,amsthm,amsfonts,amssymb,latexsym}

\usepackage[colorlinks,pagebackref]{hyperref}

\usepackage{enumerate}
\usepackage[shortlabels]{enumitem}
\usepackage{arydshln,xcolor,mathabx}

\theoremstyle{plain}

\newtheorem{thm}{Theorem}[section]
\newtheorem{lem}[thm]{Lemma}

\newtheorem{pro}[thm]{Proposition}
\newtheorem{cor}[thm]{Corollary}

\newtheorem{theoremalph}{Theorem}

\newtheorem{theorema}[theoremalph]{THEOREM}
\newtheorem{conjecturea}[theoremalph]{CONJECTURE}

\theoremstyle{definition}

\numberwithin{equation}{section}

\newcommand{\Ker}{\operatorname{Ker}}

\newcommand{\Irr}{\operatorname{Irr}}
\newcommand{\SL}{\operatorname{SL}}
\newcommand{\FF}{\mathbb{F}}

\newcommand{\sC}{\mathsf{C}}

\newcommand{\GL}{\operatorname{GL}}
\newcommand{\GU}{\operatorname{GU}}

\newcommand{\SU}{\operatorname{SU}}

\def\Q{{\mathbb Q}}
\def\irr#1{{\rm Irr}(#1)}

\def\syl#1#2{{\rm Syl}_#1(#2)}
\def\nor{\trianglelefteq\,}

\def\det#1{{\rm det}(#1)}

\def\sbs{\subseteq}

\newcommand{\Gal}{\mathrm{Gal}}

\newcommand{\Sym}{{\mathrm {Sym}}}

\newcommand{\Alt}{{\mathrm {Alt}}}

\newcommand{\CC}{{\mathbb C}}

\newcommand{\QQ}{{\mathbb Q}}
\newcommand{\ZZ}{{\mathbb Z}}

\newcommand{\ZB}{{\mathbf Z}}
\newcommand{\OB}{{\mathbf O}}
\newcommand{\EE}{{\mathbb E}}
\newcommand{\OC}{{\mathcal O}}

\newcommand{\lam}{\lambda}

\makeatletter \@namedef{subjclassname@2020}{\textup{2020}
Mathematics Subject Classification} \makeatother

\begin{document}
\title{Primes and The Field of Values of Characters}

\author[Nguyen N. Hung]{Nguyen N. Hung}
\address[Nguyen N. Hung]{Department of Mathematics, The University of Akron, Akron,
OH 44325, USA}
\email{hungnguyen@uakron.edu}

\author[Gabriel Navarro]{Gabriel Navarro}
\address[Gabriel Navarro]{Departament de Matem\`atiques, Universitat de Val\`encia, 46100 Burjassot,
Val\`encia, Spain}
\email{gabriel@uv.es}

\author[Pham Huu Tiep]{Pham Huu Tiep}
\address[Pham Huu Tiep]{Department of Mathematics, Rutgers University, Piscataway, NJ 08854, USA}
\email{tiep@math.rutgers.edu}

\thanks{The first author gratefully acknowledges the support of the AMS-Simons Research Enhancement
Grant (AWD-000167 AMS). The work of the second author is
supported by Grant PID2022-137612NB-I00 funded by MCIN/AEI/
10.13039/501100011033 and ERDF ``A way of making Europe."
 The third author gratefully acknowledges the support of the NSF (grant
 DMS-2200850) and the Joshua Barlaz Chair in Mathematics.}

\dedicatory{To J. L. Alperin, in memoriam}


\begin{abstract}
Let $p$ be a prime. For $p=2$, the fields of values of the complex
irreducible characters of finite groups whose degrees are not
divisible by $p$ have been classified; for odd primes $p$, a
conjectural classification has been proposed. In this work, we
extend this conjecture to characters whose degrees are divisible by
arbitrary powers of $p$, and we provide some evidence supporting
its validity.
\end{abstract}

\maketitle

\section{Introduction}

In \cite{ILNT}, Isaacs et al.\ discovered an unexpected connection
between the irrationalities of an irreducible character of a finite
group and the parity of its degree: if $\chi$ is an irreducible
character of odd degree and its field of values is $\QQ(\sqrt{d})$
for some square-free integer $d \neq -1$, then $d \equiv 1
\pmod{4}$. (Here, the field of values of $\chi$, denoted
$\QQ(\chi)$, is the smallest extension of $\QQ$ containing all
values of $\chi$.) This result was later generalized to arbitrary higher
irrationalities by the second and third authors in \cite{Navarro-Tiep21},
providing a complete classification of the fields of values of the
irreducible characters of odd degree. For odd primes $p$, this
description is conjectured,  remains open,  and is reduced in
\cite{Navarro-Tiep21} to a problem on quasi-simple groups.

Needless to say, the focus so far has been on characters of degree
not divisible by $p$ (or on height zero characters \cite{NRTV}, if one
introduces Brauer $p$-blocks), largely motivated by the
Galois--McKay conjecture (\cite{N04}) and its possible extensions.
An apparently innocent question raised in \cite{IN} (Question~D),
however, has shifted this perspective: could the results of
\cite{Navarro-Tiep21} also be extended to characters of degree divisible
by $p$?  Given a prime $p$, what are the field of values of the
characters divisible by a given power of $p$? The most natural
setting to look for examples and patterns is when $p=2$ in quadratic
fields $\Q(\sqrt d)$, where $d$ is a square-free integer. In all the cases
that we have checked, we observe, for instance, that if
$\Q(\chi)=\Q(\sqrt d)$ with even $d \ne \pm 2$, then $\chi(1)$ is
divisible by $4$.

Our first result solves one direction of the classification of the
abelian number fields that can occur as fields of values of
irreducible characters whose degree has a given $p$-part. Throughout
this paper, let $\nu_p$ denote the $p$-adic valuation function on
the integers; that is, $\nu_p(n)=a$ if $p^a$ is the largest power of
$p$ dividing $n$. If $\FF/\QQ$ is an abelian extension, recall that
the conductor $c(\FF)$ is the smallest integer $n\ge 1$ such that $\FF$ is
contained in the $n$-th cyclotomic field $\QQ_n:=\QQ(e^{2\pi i/n})$.
Also, if $m$ is any positive integer, we  let $\FF_m:=\FF(e^{2\pi i/m})$.
For a finite group $G$ and $\chi \in \irr G$, we define $c(\chi)$ to
be the conductor of the field $\Q(\chi)$.

\begin{theorema}\label{thm:main1}
Let $p$ be a prime, $\FF$ an abelian extension of $\QQ$, and
$a:=\nu_p(c(\FF))$. Then, for every integer $b\geq
\nu_p([\FF_{p^a}: \FF])$, there exist a finite group $G$ and
$\chi\in\Irr(G)$ such that $\QQ(\chi)=\FF$ and $\nu_p(\chi(1))=b$.
\end{theorema}

To complete the classification, we would need to prove the
following, which extends the still unproved Conjecture~B.3 of
\cite{Navarro-Tiep21} on characters of degree not divisible by $p$
to arbitrary characters.

\begin{conjecturea}\label{conj:main} Let $\chi$ be an irreducible character
of some finite group such that  $a:=\nu_p(c(\chi))\ge 1$. Then
$\chi(1)$ is divisible by $[\QQ_{p^a}(\chi): \QQ_p(\chi)]$.
\end{conjecturea}

A few remarks are in order. First, the degree of the extension
$\QQ_{p^a}(\chi)/ \QQ_p(\chi)$ can be any power of $p$ when $a$
varies. Indeed: on one hand, every abelian extension $\FF/\QQ$ can
arise as a character field (see Lemma~\ref{FG72}). On the other
hand, given a prime power $p^a$, elementary number theory shows that
there exists an extension $\FF\supseteq \QQ_{p}$ such that
$\nu_p(c(\FF))=a$ and $[\FF_{p^a}:\FF]=p^{a-1}$. Second, if
$\Q(\chi)=\Q(\sqrt d)$ with even $d \ne \pm 2$,  then
$c(\chi)=4|d|$. Hence $\nu_2(c(\chi))=3$ and $[\QQ_{2^3}(\chi):
\QQ(\chi)]=4$. This would explain our earlier observation on the
quadratic character fields and the divisibility by $4$ of character
degrees.
\smallskip

Since $[\QQ_p(\chi):\QQ(\chi)]=[\QQ_p:(\QQ_p\cap \QQ(\chi))]$, which
divides $[\QQ_p:\QQ]=p-1$, it follows that $[\QQ_{p^a}(\chi):
\QQ_p(\chi)]$ is precisely the $p$-part of $[\QQ_{p^a}(\chi):
\QQ(\chi)]$. Theorem~\ref{thm:main1} therefore implies the
following.

\begin{theorema}\label{thm:main2}
Let $\FF$ be an abelian extension of $\QQ$ and $a:=\nu_p(c(\FF))$.
Let $b\in\ZZ^{\geq 0}$. Assume that Conjecture~\ref{conj:main} holds
true. Then $\FF$ is the field of values of an irreducible character
$\chi$ of some finite group with $\nu_p(\chi(1))=b$ if and only if
\[ \nu_p([\FF_{p^a} : \FF])\leq b.\]
\end{theorema}

The next result presents evidence in support of
Conjecture~\ref{conj:main}.

\begin{theorema}\label{thm:main-evidence} Let $\chi$ be an irreducible character
of some finite group $G$. Suppose $a:=\nu_p(c(\chi))\ge 1$. Then
$\chi(1)$ is divisible by $[\QQ_{p^a}(\chi): \QQ_p(\chi)]$ if one
the following cases occurs.
\begin{enumerate}[\rm(a)]
\item $G$ is $p$-solvable;
\item $G$ is an alternating group;
\item $G$ is a general linear or general unitary group.
\item $G$ is the special linear group $\SL_n(q)$ and $p \nmid \gcd(n,q-1)$.
\item $G$ is the special unitary group $\SU_n(q)$ and $p \nmid \gcd(n,q+1)$.
\end{enumerate}
\end{theorema}

Also see Theorem \ref{main-slu} for results concerning $p'$-extensions of $\SL_n(q)$ and $\SU_n(q)$.

Next we turn to the related problem of classifying the fields of
values of the irreducible characters of finite groups whose order
has a given $p$-part (in particular, groups of $p'$-order). This
leads us to the following basic problem: is it always true that
$c(\chi)[\Q_{c(\chi)} : \Q(\chi)]$ divides $|G|$? Interestingly, the
recently proposed {\sl inductive Feit condition} (\cite{BKNT25})
implies what seemed to be an intractable problem.

 \begin{theorema}\label{thm:main-p'-groups}
 Suppose that the finite simple groups satisfy the inductive Feit condition.
 Then the following hold:
 \begin{enumerate}[\rm(a)]
 \item If $G$ is a finite group and  $\chi \in \irr G$, then $c(\chi)[\Q_{c(\chi)} : \Q(\chi)]$ divides $|G|$.

 \item Let $\FF/\QQ$ be an abelian extension, let $p$ be a prime, and $b\in\ZZ_{\ge 0}$.
 Then $\FF=\Q(\chi)$ for some $\chi \in \irr G$, where $G$ is a finite group with $\nu_p(|G|)=b$, if and only if
\[\nu_p(c(\FF)[\QQ_{c(\FF)}:\FF])\le b.\]
 \end{enumerate}
 \end{theorema}

Theorems~\ref{thm:main1}, \ref{thm:main-evidence}, and
\ref{thm:main-p'-groups} are proved in Sections~\ref{sec:2},
\ref{sec:3}, and \ref{sec:4}, respectively.

%

\section{THEOREM \ref{thm:main1}}\label{sec:2}

Our notation is mostly standard and essentially follows \cite{Isaacs1}. We begin
by recalling a result of G.\,M. Cram that will be used frequently.

\begin{lem}\label{lem:Cram}
Let $G$ be a finite group and $N\nor G$. Let $\chi\in\Irr(G)$ and
$\theta$ be an irreducible constituent of $\chi_N$. Let $\psi$ be
the Clifford correspondent of $\chi$ with respect to $\theta$. Then
$\QQ(\chi)\subseteq\QQ(\psi)$ and $[\QQ(\psi):\QQ(\chi)]$ divides
$\chi(1)/\psi(1)$.
\end{lem}

\begin{proof}
This follows from \cite[Lemma~1.1]{Cram88}.
\end{proof}

The next lemma easily implies, using direct products, that if $\FF$
is the field of values of an irreducible character $\chi$ with
$\chi(1)_p=p^b$ for $b\in\ZZ_{\ge 0}$, then $\FF$ also arises as the
field of values of a character whose degree has any $p$-part at
least $p^b$. Therefore, to classify the fields of values of
irreducible characters whose degree has a prescribed $p$-part, the
question we really have to solve is: \emph{given an abelian number
field $\FF$, what is the smallest possible $p$-part of the degree of
an irreducible character $\chi$ such that $\QQ(\chi)=\FF$}?

\begin{lem}\label{lem:4}
Let $d\in\ZZ_{\geq 1}$. Then there exists a solvable group $G$ and a
rational-valued character $\chi\in\Irr(G)$ such that $\chi(1)=d$.
\end{lem}

\begin{proof}
This is \cite[Lemma~2.1]{HungTiep23}.
\end{proof}

We also need the following.

\medskip

\begin{lem}\label{FG72} Let $n\in\ZZ_{\geq 1}$ and $\FF\supseteq \QQ$ be a subfield of
$\QQ_n$. Then there exists an irreducible character $\chi$ (of some
finite solvable group) such that $\QQ(\chi)=\FF$ and
$\chi(1)=[\QQ_n:\FF]$.
\end{lem}

\begin{proof} This is \cite[Theorem 2.2]{FG72}. See also
\cite[Theorem 2.2]{Navarro-Tiep21}.
\end{proof}

For solvable groups, Theorem~\ref{thm:main2} follows from known
results.

\begin{thm}\label{thm:solvable} Let $\FF$ be an abelian extension of $\QQ$ and $b\in\ZZ_{\geq 0}$. Then
$\FF$ is the field of values of an irreducible character $\chi$ of a
finite solvable group with $\nu_p(\chi(1))=b$ if and only if
\[ \nu_p([\QQ_{c(\FF)}:\FF])\leq b.\]
\end{thm}

\begin{proof}
The `only if' implication directly follows from Cram's theorem
(\cite[Theorem 0.1]{Cram88}). For the reverse implication, let
$a:=\nu_p([\QQ_{c(\FF)}:\FF])\leq b$. By Lemma \ref{FG72}, there
exists a solvable group $H$ and a character $\psi\in\Irr(H)$ such
that $\QQ(\psi)=\FF$ and $\psi(1)=[\QQ_{c(\FF)}:\FF]$. On the other
hand, by Lemma \ref{lem:4}, there exists a solvable group $K$ and a
rational-valued character $\varphi\in\Irr(K)$ such that
$\varphi(1)=p^{b-a}$.

Let $G:=H\times K$ and $\chi:=\psi\otimes \varphi\in\Irr(G)$. We
have $\QQ(\chi)=\QQ(\psi,\varphi)=\QQ(\psi)=\FF$ and
$\nu_p(\chi(1))=\nu_p(\psi(1)\varphi(1))=\nu_p(\psi(1))+\nu_p(\varphi(1))=a+(b-a)=b$.
\end{proof}

The next result we require is the $p'$-degree case of
Theorem~\ref{thm:main1}.

\begin{thm}[\cite{Navarro-Tiep21}, Theorem B2]\label{NT21B2}
Let $\FF$ be an abelian extension of $\QQ$ with $\nu_p(c(\FF))=a$.
Suppose that $[\QQ_{p^a}:(\FF\cap \QQ_{p^a})]$ is not divisible by
$p$. Then there exists a finite group $G$ and $\chi\in\Irr_{p'}(G)$
such that $\QQ(\chi)=\FF$.
\end{thm}

One of the key constructions in producing the desired group and
character in Theorem~\ref{NT21B2} is to reduce from a character
field of values to a smaller one where the extension is cyclic, see
\cite[Theorem 3.3]{Navarro-Tiep21}. For our purpose, we need to
generalize this to arbitrary (abelian) extensions.

\begin{thm}\label{arbitrary-extension}
Let $\chi$ be an irreducible character of a finite group $G$ and set
$\EE:=\QQ(\chi)$. Then, for any subfield $\FF$ of $\EE$, there
exists a finite group $H$ and a character $\psi\in\Irr(H)$ such that
$\QQ(\psi)=\FF$ and $\psi(1)=n\chi(1)^{n}$, where $n:=[\EE:\FF]$.
\end{thm}

\begin{proof}
Write $\mathcal{G}:=\Gal(\EE/\FF)$ as a direct product ${\sf
C}_{n_1}\times {\sf C}_{n_2}\times \cdots\times {\sf C}_{n_k}$ of
cyclic groups of orders $n_1,...,n_k$ with $n_i>1$ and $n_i$ is
divisible by $n_{i+1}$. Note that $n:=\prod_{i=1}^k
n_i=|\mathcal{G}|$. Let $\tau_i$ be a generator of ${\sf C}_i$, so
that
\[
\mathcal{G}=\{\tau_1^{a_1}\cdots \tau_k^{a_k}\mid 1\leq a_i\leq n_i
\text{ for every } 1\leq i\leq k\}
\]

Consider the $n$-iterated direct product of $G$ and view its
elements as $k$-dimensional matrices:
\[G^n:=\{X=(x_{a_1,... ,a_k})_{1\leq a_i\leq n_i \text{ for every }
1\leq i\leq k} \mid x_{a_1,...,a_k} \in G\}.\]

Let $\tau_i$ act on $G^n$ by permuting the $i$th-dimensional indices
of the entries $x_{a_1,...,a_k}$ but fixing the indices at the other
dimensions, in the following way
\[
(x_{a_1,... ,a_k})^{\tau_i} = (x_{a_1,...,a_{i-1}, a_{i}-1
,a_{i+1},...,a_k}),
\]
where two indices at the $i$th-dimension are identified if they have
the same residue modulo $n_i$. This action induces an embedding of
$\mathcal{G}={\sf C}_{n_1}\times {\sf C}_{n_2}\times \cdots\times
{\sf C}_{n_k}$ into the symmetric group $\Sym_n$. Let \[H:=G^n
\rtimes \mathcal{G}\] be the corresponding semidirect product.

Let \[\varphi:=\bigotimes_{1\leq a_i\leq n_i}\varphi_{a_1,..., a_k}
\in\Irr(G^n),\] where
\[
\varphi_{a_1,...,a_k}=\chi^{\tau_1^{a_1}\cdots
\tau_k^{a_k}}\in\Irr(G);
\]
that is,
\[
\varphi(X)= \prod_{1\leq a_i\leq n_i \atop 1\leq i\leq k}
\chi^{\tau_1^{a_1}\cdots \tau_k^{a_k}}(x_{a_1,...,a_k}).
\] The action of a typical element $\tau=\tau_1^{b_1}\cdots \tau_k^{b_k}$
of $\mathcal{G}$ on $\varphi$ is given by
\[
(\varphi^\tau)_{a_1,...,a_k}=\chi^{\tau_1^{a_1+b_1}\cdots
\tau_k^{a_k+b_k}}.
\]
It is clear that $\varphi^\tau\neq \varphi$ whenever $\tau\neq 1$.
Therefore, the stabilizer $I_H(\varphi)$ of $\varphi$ in $H$ is
simply $G^n$, which in turn implies that
\[\psi:=\varphi^H\in\Irr(H).\] We have
$\psi(1)=[H:G^n]\varphi(1)=n\chi(1)^n$. The rest of the proof will
establish that $\QQ(\psi)=\FF$. This can be done by an suitable
modification of the arguments in the proof of \cite[Theorem
3.3]{Navarro-Tiep21}.

By the definition of $\varphi$, we know
$\QQ(\varphi)=\QQ(\chi)=\EE$. The character-induction formula then
implies that $\QQ(\psi)\subseteq \EE$. Furthermore, for each $X\in
G^n$ and $j\in \{1,...,k\}$,
\[
\varphi(X^{\tau_j})= \prod_{1\leq a_i\leq n_i \atop 1\leq i\leq k}
\chi^{\tau_1^{a_1}\cdots
\tau_k^{a_k}}(x_{a_1,...a_{j-1},a_j-1,a_{j+1},...,a_k}).
\]
On the other hand,
\[
\varphi^{\tau_j}(X)= \prod_{1\leq a_i\leq n_i \atop 1\leq i\leq k}
\chi^{\tau_1^{a_1}\cdots\tau_{j-1}^{a_{j-1}}\tau_j^{a_j+1}\tau_{j+1}^{a_{j+1}}\cdots
\tau_k^{a_k}}(x_{a_1,...,a_k}).
\]
If all the $a_1,...,a_{j-1},a_{j+1},...,a_k$ are fixed, the
sub-products along the $j$th dimension are equal:
\[
\prod_{1\leq a_j\leq n_j} \chi^{\tau_1^{a_1}\cdots
\tau_k^{a_k}}(x_{a_1,...a_{j-1},a_j-1,a_{j+1},...,a_k})=
\prod_{1\leq a_j\leq n_j}
\chi^{\tau_1^{a_1}\cdots\tau_{j-1}^{a_{j-1}}\tau_j^{a_j+1}\tau_{j+1}^{a_{j+1}}\cdots
\tau_k^{a_k}}(x_{a_1,...,a_k})
\]
It follows that \[ \varphi(X^{\tau_j})=\varphi^{\tau_j}(X).
\]
Since $\mathcal{G}$ is generated by $\tau_j, 1\leq j\leq k$, we
conclude that
\[
\varphi(X^{\tau})=\varphi^{\tau}(X)
\]
for every $\tau\in \mathcal{G}$. That is, the Galois action of
$\mathcal{G}$ on $\varphi$ is compatible with our defined action of
$\mathcal{G}$ on $G^n$.

Now, for every $X\in G^n$ and $\tau\in \mathcal{G}$, by the Clifford
theorem, we have
\begin{align*}
\psi(X)^\tau = &\left(\sum_{\tau'\in \mathcal{G}}
\varphi(X^{\tau'})\right)^\tau= \left(\sum_{\tau'\in \mathcal{G}}
\varphi(X)^{\tau'}\right)^\tau\\
 = &\sum_{\tau'\in \mathcal{G}}
\varphi(X)^{\tau\tau'}=\sum_{\tau'\in \mathcal{G}}
\varphi(X)^{\tau'}=\psi(X).
\end{align*}
Therefore, $\psi(X)$ is fixed by $\mathcal{G}$ point-wisely. As
$\psi$ vanishes outside $G^n$, we deduce that
$\QQ(\psi)=\QQ(\psi_{G^n})$ is contained in the fixed field of
$\mathcal{G}$ in $\EE$. That is, $\QQ(\psi)\subseteq \FF$.

By Lemma~\ref{lem:Cram}, we have $[\QQ(\varphi):\QQ(\psi)]$ divides
$[H:G^n]=n$. As mentioned already that $\QQ(\varphi)=\EE$, it
follows that $[\EE:\QQ(\psi)]\leq n=[\EE:\FF]$. This and the
conclusion of the preceding paragraph imply that $\QQ(\psi)=\FF$.
The proof is complete.
\end{proof}

We can now complete the proof of Theorem~\ref{thm:main1}, which is
restated.

\begin{thm}
Let $\FF$ be an abelian extension of $\QQ$ and $a:=\nu_p(c(\FF))$.
Let $c:=\nu_p([\FF_{p^a}: \FF])$. Then, for every integer $b\geq
c$, there exists a finite group $G$ and a character $\chi\in\Irr(G)$
such that $\QQ(\chi)=\FF$ and $\nu_p(\chi(1))=b$.
\end{thm}

\begin{proof}
Let $\EE:=\QQ_{p^a}(\FF)$. It is clear that $\nu_p(c(\EE))=a$ and
$[\QQ_{p^a}:(\EE\cap \QQ_{p^a})]=1$. Using Theorem~\ref{NT21B2}, we
obtain a finite group $H$ and an irreducible character $\psi$ of $H$
of $p'$-degree such that $\QQ(\psi)=\EE$.

By the hypothesis, we have $[\FF_{p^a}: \FF]=[\EE:\FF]=p^cf$ for
some $f\in \ZZ^{+}$ not divisible by $p$. Now applying Theorem
\ref{arbitrary-extension} to the extension $\EE/\FF$, we then obtain
a finite group $M$ and $\varphi\in\Irr(M)$ such that
\[\varphi(1)=p^cf\psi(1)^{p^cf} \text{ and } \QQ(\varphi)=\FF.\]

Next, we use Lemma \ref{lem:4} to have a finite group $N$ and some
rational-valued character $\lambda\in\Irr(N)$ such that
$\lambda(1)=p^{b-c}$.

Finally, let $G:=M\times N$ and $\chi:=\varphi\otimes \lambda$ -- an
irreducible character of $G$. This $\chi$ is indeed the required
character. To see that, we have
\[\chi(1)=\varphi(1)\lambda(1)=p^cf\psi(1)^{p^cf}\cdot p^{b-c}=p^{b}f\psi(1)^{p^cf}.\]
Recall that both $f$ and $\psi(1)$ are not divisible by $p$.
Therefore $\nu_p(\chi(1))=b$. Moreover,
\[
\QQ(\chi)=\QQ(\varphi,\lambda)=\QQ(\varphi)=\FF,
\]
as desired.
\end{proof}


\section{Theorem \ref{thm:main-evidence}}\label{sec:3}

We begin with the following lemma about $p$-groups.

\begin{lem}\label{pgroups}
Suppose that $P$ is a $p$-group. Let $\alpha \in \irr P$ be with $c(\alpha)=p^a$ and $a\ge 1$.
\begin{enumerate}[\rm(a)]
\item
If $p$ is odd or $p^a=4$, then $\Q_{p^a}(\alpha)=\Q_p(\alpha)$.

\item
If $p=2$ and $p^a \ge 8$, then $|\Q_{p^a}(\alpha):\Q_p(\alpha)| \le 2$ divides $\alpha(1)$.
\end{enumerate}
\end{lem}

\begin{proof}
If $p$ is odd, then we know that $\Q(\alpha)=\Q_{p^a}$ by
\cite[Theorem~2.3]{Navarro-Tiep21}. If $p=2$ and $a=2$, then
$\Q(\alpha) \sbs \Q(i)$ and necessarily $\Q(\alpha)=\Q(i)$, so there
is nothing to prove. Assume now that $p=2$ and $a \ge 3$. By
elementary Galois theory, we know that $\Q_{2^a}/\Q_8$ is a cyclic
extension with intermediate fields $\Q_8 \sbs \Q_{16} \sbs  \ldots
\sbs  \Q_{2^a}$. Therefore $\Q_8(\alpha)=\Q_{2^a}=\Q_{2^a}(\alpha)$.
Also $\Q_8 \cap \Q(\alpha)$ cannot be $\Q$, since $\Q_{2^a}$ has
only three subfields of degree 2 and all of them are contained in
$\Q_8$. If $\alpha$ is linear, notice that
$[\Q_{p^a}(\alpha):\Q_p(\alpha)]=1$. Part (b) easily follows.
\end{proof}

\begin{thm}\label{thmD-solvable}
Let $G$ be a $p$-solvable finite group, and let $\chi \in \irr G$ be
with $a=\nu_p(c(\chi))\ge 1$. Then $[\Q_{p^a}(\chi):\Q_p(\chi)]$
divides $\chi(1)$.
\end{thm}

\begin{proof} Notice that $[\Q_{p^a}(\chi):\Q_p(\chi)]$ is a power of $p$.
We argue by induction on $|G|$. Suppose that there is $N\nor G$, and
$\eta^G=\chi$, where $\eta \in \irr V$ is a Clifford correspondent
of $\chi$. By Lemma~\ref{lem:Cram}, we know that
$[\Q(\eta):\Q(\chi)]$ divides $|G:V|$. We have that $c(\chi)$
divides $c(\eta)$. Write $c(\eta)=p^b$, where $b\ge 1$. By
induction, we have that $[\Q_{p^a}(\eta):\Q_p(\eta)]$ divides
$[\Q_{p^b}(\eta):\Q_p(\eta)]$ divides $\eta(1)$. If
$K=\Q_{p^a}(\chi)\cap \Q_p(\eta)$, by elementary Galois theory, we
have that $[K:\Q_p(\chi)]$ divides $[\Q(\eta):\Q(\chi)]$ which
divides $|G:V|$, and we are done.

We may now assume that $\chi$ is quasi-primitive. Hence,
$\chi=\alpha\beta$, where $\alpha$ is $p$-special and $\beta$ is
$p'$-special (see \cite[Theorem~2.9]{Isaacs2018}). We claim that
$[\Q_{p^a}(\chi):\Q_p(\chi)]=1$, in this case. First, we prove that
$\Q(\chi)=\Q(\alpha,\beta)$. Certainly, $\Q(\chi) \sbs \Q(\alpha,
\beta)$. However, if $\sigma \in {\rm Gal}(\Q(\alpha,
\beta)/\Q(\chi)$, then we obtain that $\alpha^\sigma=\alpha$ and
$\beta^\sigma=\beta$, by the uniqueness of the product
decomposition. Since $\Q(\alpha)=\Q(\alpha_P)$, where $P \in \syl
pG$ (by the uniqueness of the restriction from $p$-special
characters), we have that $c(\alpha\beta)=c(\alpha)c(\beta)$. If $p$
is odd or $p^a=4$, then $\Q_{p^a}=\Q_p(\alpha)$, by Lemma
\ref{pgroups}(a). Then $\Q_{p^a}(\chi)=\Q_{p^a}(\alpha,
\beta)=\Q_{p^a}(\beta)$ and $\Q_{p}(\chi)=\Q_{p}(\alpha,
\beta)=\Q_{p^a}(\beta)$. Suppose finally that $p=2$ and $a \ge 3$.
If $\chi$ has odd-degree, then the theorem follows from
\cite[Theorem~A1]{Navarro-Tiep21}. So we may assume that $\chi(1)$
is even. By Lemma \ref{pgroups}, we have that
$[\Q_{p^a}(\chi):\Q_p(\chi)] \le 2$, and the theorem follows.
\end{proof}

\begin{thm}\label{thmD-alternating}
Let $n\geq 1$ and $\chi \in \irr {\Alt_n} \cup \irr {\Sym_n}$. Then
$\nu_p(c(\chi))\leq 1$. In particular, Theorem
\ref{thm:main-evidence} holds for the symmetric and alternating
groups.
\end{thm}

\begin{proof}
It is well known that the characters of $\Sym_n$ are all rational,
so the result is immediate in this case. Hence, we may assume that
$\chi \in \Irr(\Alt_n)$. For background on the ordinary irreducible
characters of $\Alt_n$, we refer the reader to \cite[\S5]{JK}.

Suppose that $\Psi$ is the irreducible character of $\Sym_n$ lying
over $\chi$, and let $\lambda$ be the partition of $n$ labeling
$\Psi$. If $\lambda$ is not self-conjugate, then $\chi =
\Psi_{\Alt_n}$ and the result is again trivial. Thus we assume that
$\lambda$ is self-conjugate. In this case, $ \Psi_{\Alt_n} = \chi +
\chi^\sigma, $ where $\sigma$ is an odd permutation in $\Sym_n$.

Let $h(\lambda)=(h_{11}^\lambda,\dots,h_{kk}^\lambda)$ denote the
partition of $n$ whose parts $h_{ii}^\lambda$ are the hook lengths
at the diagonal nodes $(i,i)$ of the Young diagram of $\lambda$.
Here $k$ is the length of the main diagonal. The conjugacy class of
$\Sym_n$ of cycle type $h(\lambda)$ splits into two conjugacy
classes of $\Alt_n$, say $h(\lambda)^+$ and $h(\lambda)^-$.

According to \cite[Theorem~2.5.13]{JK}, the values of $\chi$ (and of
$\chi^\sigma$) are rational at all elements, except possibly at
those belonging to $h(\lambda)^\pm$. In fact, one has
\[
\Q(\chi)=\Q\!\left(\sqrt{(-1)^{(n-k)/2}\,h_{11}^\lambda \cdots
h_{kk}^\lambda}\right).
\]

Let $d$ denote the square-free part of $h_{11}^\lambda \cdots
h_{kk}^\lambda$, that is, the product of the prime divisors
occurring with odd exponents. Note that all the hook lengths
$h_{ii}^\lambda$ are odd and that
\[
n-k = \sum_{i=1}^k (h_{ii}^\lambda-1)
\]
is even. We shall use $c(z)$ to denote the conductor of a sum of
roots of unity $z$, which is the smallest positive integer $t$ such
that $z\in \QQ_t$. If $n-k \equiv 0 \pmod{4}$, then
  $h_{11}^\lambda \cdots h_{kk}^\lambda \equiv 1 \pmod{4}$, and hence
  \[
  c(\chi) = c\!\left(\sqrt{h_{11}^\lambda \cdots h_{kk}^\lambda}\right)
  = c(\sqrt{d}) = d.
  \]
On the other hand, if $n-k \equiv 2 \pmod{4}$, then
  $h_{11}^\lambda \cdots h_{kk}^\lambda \equiv 3 \pmod{4}$, and hence
  \[
  c(\chi) = c\!\left(\sqrt{-h_{11}^\lambda \cdots h_{kk}^\lambda}\right)
  = c(\sqrt{-d}) = d.
  \]
It follows that $\nu_p(c(\chi))=1$ if $p$ divides $d$, and
$\nu_p(c(\chi))=0$ otherwise. This completes the proof.
\end{proof}

We now prove Theorem \ref{thm:main-evidence} for finite general
linear groups $G=\GL_n(q)$, where $n \geq 1$ and $q$ is a prime
power, with a natural module $V = \FF_q^n = \langle e_1, \ldots ,e_n
\rangle_{\FF_q}$. It is convenient for us to use the Dipper-James
classification of complex irreducible characters of $G$, as
described in \cite{JamesGL}. Namely, every $\chi\in \Irr(G)$ can be
written uniquely, up to a permutation of the pairs
$\{(s_1,\lam_1),\dots,(s_m,\lam_m)\}$, in the form
\begin{equation}\label{gl10}
  \chi = S(s_1,\lam_1) \circ S(s_2,\lam_2) \circ \ldots \circ S(s_m,\lam_m).
\end{equation}
Here, $s_i \in \overline{\FF_q}^\times$ has degree $d_i$ over
$\FF_q$, $\lam_i \vdash k_i$, $\sum^m_{i=1}k_id_i = n$, and the $m$
elements $s_i$ have pairwise distinct minimal polynomials over
$\FF_q$. In particular, $S(s_i,\lam_i)$ is an irreducible character
of $\GL_{k_id_i}(q)$. Furthermore, there is a parabolic subgroup
$P_\chi = U_\chi L_\chi$ of $G$ with Levi subgroup $L_\chi =
\GL_{k_1d_1}(q) \times \ldots \times \GL_{k_md_m}(q)$ and unipotent
radical $U_\chi$. The (outer) tensor product
$$\psi := S(s_1,\lam_1) \otimes S(s_2,\lam_2) \otimes \ldots \otimes S(s_m,\lam_m)$$
is an $L_\chi$-character, and $\chi$ is obtained from $\psi$ via the
Harish-Chandra induction $R^{G}_{L_\chi}$, i.e. we first inflate
$\psi$ to a $P_\chi$-character and then induce it to $G$.

\begin{lem}\label{act1}
In the above notation, assume that the integer $N \geq 1$ is
divisible by the order of every $s_i$, $1 \leq i \leq m$, and that
$\sigma \in \Gal(\QQ_N/\QQ)$ sends a primitive $N^{\mathrm{th}}$
root of unity $\zeta \in \CC^\times$ to $\zeta^k$, $1 \leq k < N$.
Then for the character $\chi$ in \eqref{gl10},
$$\sigma(\chi) = S(s_1^k,\lam_1) \circ S(s_2^k,\lam_2) \circ \ldots \circ S(s_m^k,\lam_m).$$
\end{lem}

\begin{proof}
First we show that $s_i^k$ also has degree $d_i$ over $\FF_q$, for
any $i$. Indeed, $s_i \in \FF_{q^{d_i}}$ implies that $s_i^k \in
\FF_{q^{d_i}}$, and hence $s_i^k$ has degree $d'|d_i$. Conversely,
as $\gcd(k,N)=1$, we can find $l \in \ZZ_{\geq 1}$ such that
$N|(kl-1)$. Now $s_i=s_i^{kl} \in \FF_{q^{d'}}$, and so $d_1|d'$,
whence $d'=d_i$.

Since the Harish-Chandra induction $R^G_{L_\chi}$ commutes with
Galois action, it suffices to prove the claim in the case $m=1$.
Now, if $d_1 > 1$, then $\chi=S(s_1,\lambda_1)$ is Lusztig induced
from the character $S(s_1,\lambda_1)$ of the Levi subgroup $L=
\GL_{k_1}(q^{d_1})$ of $G=\GL_{k_1d_1}(q)$, see \cite[p. 116]{FS}.
Since the Lefschetz numbers are all integers (see e.g. \cite[p.
140]{DM}), Green functions involved in the Lusztig induction $R^G_L$
are rational-valued, and thus $R^G_L$ also commutes with Galois actions.
Thus it suffices to prove the claim in the case $m=d_1=1$.  In the
latter case, $\chi=S(1,\lambda_1)\hat{s_1}$, where $S(1,\lambda_1)$
is a rational (unipotent) character, and $\hat{s_1}$ is a linear
character of $G/[G,G]$ of order equal to the order of $s_1$, see
\cite[(1.16)]{FS}, so we are done.
\end{proof}

\begin{lem}\label{sgr}
Let $p$ be a prime and $A = {\sf C}_{p^n} \times {\sf C}_p$ for some
$n \geq 1$. If $B$ is a subgroup of $A$, then there is some $1 \leq
m \leq n$ such that $B \cong \sC_{p^m}$ or $B \cong \sC_{p^m} \times
\sC_p$.
\end{lem}

\begin{proof}
Write $A = A_1 \times A_2$ with $A_1 = \langle x \rangle \cong
\sC_{p^n}$ and $A_2 = \langle y \rangle \cong \sC_p$, and consider
the natural projection $\pi:A \to A_1$ with kernel $A_2$. If
$\pi|_B$ is injective then $B$ is cyclic of order $p^m$ with $m \leq
n$, and we are done. Otherwise $1 \neq \Ker(\pi|_B)$, so $y \in B$,
whence $B = A_2 \times (A_1 \cap B)$, and we are done again.
\end{proof}

\begin{thm}\label{main-gl}
Theorem \ref{thm:main-evidence} holds for finite general linear
groups $G = \GL_n(q)$ and finite general unitary groups $\GU_n(q)$.
\end{thm}

\begin{proof}
Since the statement is trivial when $a=1$ or $n=1$, we may assume
that $a,n \geq 2$.

The case of $\GU_n(q)$ is entirely similar to that of $\GL_n(q)$,
changing only $q$ to $-q$  (in character degrees and subgroup
orders, as well as in the definition of $[t]$ for $t \in
\overline{\FF_q}^\times$) and Harish-Chandra induction to Lusztig
induction. So we will provide the details only for $G = \GL_n(q)$.
Note that $p \nmid q$ since \eqref{gl10} shows that
$c(\chi)$ divides $N :=\mathrm{lcm}(o(s_1), \ldots ,o(s_m))$ which divides
$\prod^n_{i=1}(q^i-1)$. So $q$ is power of some prime $r
\neq p$.

\smallskip
(a) Here we show that if a $p$-element $\sigma \in
\Gal(\QQ_N/\QQ_p(\chi))$ induces an element of
$\Gal(\QQ_{p^a}(\chi)/\QQ_p(\chi))$ of order $p^{a_1}$ for some $a_1
\geq 0$, then $\chi(1)$ is divisible by $p^{a_1}$, in fact by
$2^{a_1+1}$ if $p=2 \leq a_1$.

The case $a_1=0$ is trivial. so we may assume $a_1 \geq 1$, and note
that $\chi$ is $\sigma$-invariant. Writing $\chi$ in the form
\eqref{gl10} and applying Lemma \ref{act1} with
$\sigma(\zeta)=\zeta^f$, we see that the induced map $\tau:
S(s,\lam) \mapsto S(s^f,\lam)$ acts on the set $S$ of the characters
$S(s_i,\lam_i)$, $1 \leq i \leq m$.

\smallskip
(a1) First we show that if some $\tau$-orbit $\OC$ on $S$ has length
$p^b$ with $b \geq 0$, say $S(s_1,\lam_1) \in \OC$, and $p$ divides
$q^k-1$ with $k:=k_1d_1$, then both
$$N_1:=\frac{|\GL_{kp^b}(q)|_{r'}}{|\GL_{k}(q)|_{r'}^{p^b}}$$
and $\chi(1)$ are divisible by $p^b$, in fact by $p^{b+1}$ if $p = 2
\leq b$.

Indeed, the statement is obvious if $b=0$. Assume $b \geq 1$. Then
for any $1 \leq i \leq b$, $p|(q^{kp^{i}}-1)/(q^{kp^{i-1}}-1)$,
hence $p^b$ divides $(q^{kp^{b}}-1)/(q^{k}-1)$. Since
$$N_1:= N_2 \cdot \frac{q^{kp^b}-1}{q^k-1}, \mbox{ where }N_2:= \frac{|\GL_{kp^b-1}(q)|_{r'}}{|\GL_{k}(q)|_{r'}^{p^b-1}|\GL_{k-1}(q)|_{r'}}$$
is an integer, $p^b|N_1$. Next, $N_1$ divides
$$\frac{|\GL_n(q)|_{r'}}{\prod^m_{i=1}|\GL_{k_id_i}(q)|_{r'}}.$$
So $N_1|\chi(1)$, and hence $p^b|\chi(1)$.

Assume in addition that $p=2 \leq b$. Then we can express $N_2$ as
$$\biggl(\frac{|\GL_{k\cdot 2^{b-1}-1}(q)|_{r'}}{|\GL_{k}(q)|_{r'}^{2^{b-1}-1}|\GL_{k-1}(q)|_{r'}}\biggr)^2\cdot
    \frac{|\GL_{k\cdot 2^{b}-1}(q)|_{r'}}{|\GL_{k\cdot 2^{b-1}}(q)|_{r'}|\GL_{k\cdot 2^{b-1}-1}(q)|_{r'}}
 \cdot \frac{q^{k \cdot 2^{b-1}}-1}{q^k-1},$$
where the first two factors are integers, and the third one is an
even integer. So $2^{b+1}$ divides both $N_1$ and $\chi(1)$.

\smallskip
(a2) For any $t \in \overline{\FF_q}^\times$ of degree say $l$ over
$\FF_q$, let $[t]$ denote the set $\{t,t^q, \ldots,t^{q^{l-1}}\}$.
Suppose that some $\tau$-orbit $\OC$ on $S$ has length $p^b$ with $b
\geq 0$, say $S(s_1,\lam_1) \in \OC$, and $p|(q^d-1)$ with $d:=d_1$.
As $\tau^{p^b}$ fixes $S(s_1,\lam_1)$,  $\tau^{p^b}$ acts on the set
$[s_1]$. Suppose in addition
that some $\tau^{p^b}$-orbit on $[s_1]$ has length $p^c$ for some $c \geq 0$. Then we
show that $p^{b+c}|\chi(1)$, and in fact $2^{b+c+1}|\chi(1)$ if $p =
2 \leq b+c$.

Indeed, write $k:=k_1$. If $c=0$, then we are done by (a1). So we
assume that $c \geq 1$. By (a1),
$$N_1:=\frac{|\GL_{kdp^b}(q)|_{r'}}{|\GL_{kd}(q)|_{r'}^{p^b}}$$
is divisible by $p^b$, in fact by $p^{b+1}$ if $p=2 \leq b$.

Since the maps $\tau^{p^b}: x \mapsto x^{f^{p^b}}$ and $x \mapsto
x^q$ have commuting actions on $[s_1]$, with the latter acting
transitively, we see that $p^c$ divides the size of any
$\tau^{p^b}$-orbit on $[s_1]$, which has cardinality $d$. Thus
$p^c|d$; write $d=p^cd_1$. As $p \nmid q$, we have $q^d \equiv
q^{d_1} \pmod{p}$, and so $p|(q^{d_1}-1)$. Hence $p$ divides each
$q^{id_1}-1$ with $1 \leq i \leq p^c-1$, and so $p^{p^c-1}$ divides
$(q-1)(q^2-1) \ldots (q^{d-1}-1)$, which is a divisor of
$$\frac{(q-1)(q^2-1) \ldots (q^{kd}-1)}{(q^d-1)(q^{2d}-1) \ldots (q^{kd}-1)}.$$
The latter divides the degree of $S(s_1,\lambda_1)$, see \cite[Lemma
5.7(ii)]{KT2}. Also, $c \geq 1$ implies that $p^c-1 \geq 2^c-1 \geq
c$. It follows that the degree of $S(s_1,\lam_1)$ is divisible by
$p^c$, in fact by $p^{c+1}$ if $p=2 \leq c$.

Thus $p^{b+c}|\chi(1)$; moreover $p^{b+c+1}|\chi(1)$ if $p=2 \leq
\max(b,c)$. It remains to consider the case $p=2$ and $b=c=1$. In
this case $\chi(1)$ is divisible by
$N_1\bigl(\deg(S(s_1,\lambda_1))\bigr)^2$, so by $2^3$ as stated.

\smallskip
(a3) Now, if $\tau$ has some orbit of length $p^b$ on $S$ containing
$S(s_j,\lam_j)$ with $p|o(s_j)$ and $\tau^{p^b}$
has some orbit of length
$p^c$ on the set $[s_j]$ with $b+c \geq a_1$, then we are done by
(a2). So we may assume that $\tau_1:=\tau^{p^{a_1-1}}$ fixes every
$s_i$ with $p|o(s_i)$. Thus if $\tau_1(x)=x^{f_1}$, then $f_1-1$ is
divisible by $o(s_i)$ whenever $p|o(s_i)$. On the other hand,
\eqref{gl10} shows that $c(\chi)$ divides $\mathrm{lcm}(o(s_1),
\ldots,o(s_m))$. It follows that $p^a=c(\chi)_p$ divides $f_1-1$.
But in this case, $\tau_1$ acts trivially on both $\QQ_{p^a}$ and
$\QQ(\chi)$, contrary to the assumption that $\tau_1$ has order $p$
in $\Gal(\QQ_{p^a}(\chi)/\Q_p(\chi))$.

\smallskip
(b) By \cite[Theorem 3.4.3]{We}, the restriction to $\QQ_{p^a}$
gives an injective homomorphism $\Gal(\QQ_{p^a}(\chi)/\QQ_p(\chi))
\to \Gal(\QQ_{p^a}/\QQ_p)$, and the latter is isomorphic to
$\sC_{p^{a-1}}$ if $p > 2$ and to $\sC_{2^{a-2}} \times \sC_2$ if
$p=2$. It follows that $|\Gal(\QQ_{p^a}(\chi)/\QQ_p(\chi))|=p^e$
with $e \leq a-1$.

It suffices to prove the statement in the case $e \geq 1$. Now, if
$\Gal(\QQ_{p^a}(\chi)/\QQ_p(\chi))$ is cyclic, then we can find an
element of order $p^e$ in $\Gal(\QQ_{p^a}(\chi)/\QQ_p(\chi))$, and
so $p^e|\chi(1)$ by (a), and we are done.

By Lemma \ref{sgr}, it remains to consider the case $p=2 \leq e$ and
$\Gal(\QQ_{p^a}(\chi)/\QQ_p(\chi)) \cong \sC_{2^{e-1}} \times
\sC_2$. Now, if $e \geq 3$, then then we can find an element of
order $2^{e-1} \geq 4$ in $\Gal(\QQ_{p^a}(\chi)/\QQ_p(\chi))$, and
so $2^{e+1}|\chi(1)$ by (a), and we are done again. Suppose now that
$e=2$ but $4 \nmid \chi(1)$.  The analysis in (a) then shows that,
for every $\tau$ induced by a nontrivial element in
$\Gal(\QQ_{p^a}(\chi)/\QQ_p(\chi))$, $\tau$ has exactly one orbit of
length $2$ on
$$S':= \{(s_i^{q^j},\lam_i) \mid 1 \leq i \leq m, 0 \leq j \leq d_j-1\},$$
with all other orbits having length $1$. However, since $4 \nmid
\chi(1)$, the same holds for the induced action of
$\Gal(\QQ_{p^a}(\chi)/\QQ_p(\chi)) = \langle
\sigma_1,\sigma_2\rangle\cong \sC_2^2$. Thus if $\tau_i$ denotes the
induced action of $\sigma_i$, $i = 1,2$, then we may assume that
both $\tau_{1,2}$ act as the same $2$-cycle on $S'$. But then the
automorphism $\sigma_1\sigma_2$ of order $2$ acts trivially on $S'$,
a contradiction.
\end{proof}

We will need the following refinement of \cite[Lemma 4.2]{Navarro-Tiep21}:

\begin{pro}\label{pro:ext1}
Let $p$ be a prime and let $N$ be a normal subgroup of a finite group $G$ with $p \nmid |G/N|$.
Let $\theta \in \Irr(N)$, $J:=\mathrm{Stab}_G(\theta)$, and let $\FF$ by any extension of $\QQ$ inside $\QQ_{|G|}$.
Then the following statements hold.
\begin{enumerate}[\rm(i)]
\item Suppose that $\chi \in \Irr(G|\theta)$, $a = \nu_p(c(\chi))$, and $b=\nu_p(c(\theta))$. If $a,b \geq 1$
or $p=2$, then $a=b$. If $a =0$ then $b \leq 1$, and if $b=0$ then $a \leq 1$. Furthermore,
$p \nmid [\FF(\theta):(\FF(\chi) \cap \FF(\theta))]$. In fact, every prime divisor of
$[\FF(\theta):(\FF(\chi_N) \cap \FF(\theta))]$ divides $|G/N|$.
\item Suppose that all the characters in $\Irr(J|\theta)$ have the same degree (this holds if, for instance, $J/N$ is cyclic). Then,
there exists some $\chi \in \Irr(G|\theta)$ such that $p \nmid [\FF(\chi):(\FF(\chi) \cap \FF(\theta))]$.
\end{enumerate}
\end{pro}

\begin{proof}
(i) The first statement is just \cite[Lemma
4.2(ii)]{Navarro-Tiep21}. The second statement follows from
\cite[Lemma~2.4(ii)]{Hung22}, where it was shown that if
$\max\{a,b\}\geq 2$, then $a=b$.

For the third statement, by Clifford's theorem we have
$\chi_N = e\sum^t_{i=1}\theta_i$, where $t=[G:J]$ is coprime to $p$, and $\theta=\theta_1, \theta_2, \ldots,\theta_t$ are the $t$ distinct
$G$-conjugates of $\theta$; in particular, $\QQ(\theta_i)=\QQ(\theta)$ for all $i$. Now let $\Gamma := \Gal(L/K)$, where
$L :=\FF(\chi,\theta)$ and $K:=\FF(\chi)$, and let
$\Gamma_p$ denote the Sylow $p$-subgroup of $\Gamma$. Then, any
$\sigma \in \Gamma_p$ fixes $\chi_N$, and so it permutes
$\theta_1,...,\theta_t$. Thus the $p$-group
$\Gamma_p$ acts on the set $\{\theta_1, \ldots,\theta_t\}$
of $p'$-size, and hence $\Gamma_p$ fixes
some $\theta_j$. Setting $M:=\FF(\theta_j) = \FF(\theta)$, we see that $MK$ is
contained in the fixed field $L^{\Gamma_p}$ which has $p'$-degree
$[\Gamma:\Gamma_p]$ over $K$ by Galois correspondence. It follows
that $p \nmid [MK:K] = [M:K \cap M] = [\FF(\theta):(\FF(\chi) \cap
\FF(\theta))]$.

The same proof using Sylow subgroups of $\Gal(\FF(\theta,\chi_N)/\FF(\chi_N))$ yields the fourth statement.

\smallskip
(ii) The assumption implies that all the characters $\gamma$ in $\Irr(J|\theta)$ have the same ramification index $e$:
$\gamma_N=e\theta$, and so, by the Clifford correspondence \cite[(6.11)]{Isaacs1}, if $\Irr(G|\theta)= \{\chi_1, \ldots,\chi_s\}$ then
$$[\theta^G,\chi_i]_G=[\theta,(\chi_i)_N]_N=e$$
for all $i$. Hence,
\begin{equation}\label{ext20}
  \theta^G = e\sum^s_{i=1}\chi_i,
\end{equation}
and thus
$$\bigl(\theta^G\bigr)_N=e\sum^s_{i=1}(\chi_i)_N = e^2s\sum^t_{j=1}\theta_j$$
in the notation of (i). It follows that $e^2st=|G/N|$, and so $p \nmid s$.

Now let $\Gamma := \Gal(L/K)$, where $L :=
\FF(\chi_1,\ldots,\chi_s,\theta)$ and $K:=\FF(\theta)$, and let
$\Gamma_p$ denote the Sylow $p$-subgroup of $\Gamma$. Then, any
$\sigma \in \Gamma_p$ fixes $\theta^G$, and so it permutes
$\chi_1,...,\chi_s$ by \eqref{ext20}. Thus the $p$-group
$\Gamma_p$ acts on the set $\{\chi_1, \ldots,\chi_s\}$
of $p'$-size, and hence $\Gamma_p$ fixes
some $\chi=\chi_i$. Setting $M:=\FF(\chi)$, we see that $MK$ is
contained in the fixed field $L^{\Gamma_p}$ which has $p'$-degree
$[\Gamma:\Gamma_p]$ over $K$ by Galois correspondence. It follows
that $p \nmid [MK:K] = [M:K \cap M] = [\FF(\chi):(\FF(\chi) \cap
\FF(\theta))]$.
\end{proof}

Now we can prove the following ``going-up and going-down'' result:

\begin{thm}\label{thm:ext}
Let $p$ be a fixed prime, $G$ a finite group, and $N$ a normal subgroup of $G$ such that $p \nmid |G/N|$.
Then the following statements hold.
\begin{enumerate}[\rm(i)]
\item Suppose Conjecture \ref{conj:main} holds for all irreducible characters of $N$ at the prime $p$.
Then Conjecture \ref{conj:main} holds for all irreducible characters of $G$ at the prime $p$.
\item Suppose that $G/N$ is solvable and
that Conjecture \ref{conj:main} holds for all irreducible characters of $G$ at the prime $p$.
Then Conjecture \ref{conj:main} holds for all irreducible characters of $N$ at the prime $p$.
\end{enumerate}
\end{thm}

\begin{proof}
(i) It suffices to prove the theorem in the case $\chi \in \Irr(G)$ has
$a:=\nu_p(c(\chi)) \geq 2$. By Proposition \ref{pro:ext1}(i), for any $\theta \in \Irr(N)$ lying under $\chi$ we have
$$\nu_p(c(\theta))=a,$$
furthermore,
$$p \nmid [M:K]$$
with $M:=\QQ_p(\theta)$ and $K:= \QQ_p(\chi) \cap \QQ_p(\theta)$.
Since
$$\begin{aligned}[\QQ_{p^a}M:\QQ_{p^a}K] & = [\QQ_{p^a}M:\QQ_{p^a}]/[\QQ_{p^a}K:\QQ_{p^a}]\\
    & = [M:(\QQ_{p^a}\cap M)]/[K:(\QQ_{p^a} \cap K)]\\
    & = [M:K]/[(\QQ_{p^a}\cap M):(\QQ_{p^a} \cap K)]\end{aligned}$$
is an integer, it follows that
$$p \nmid [(\QQ_{p^a}\cap \QQ_p(\theta)):(\QQ_{p^a} \cap \QQ_p(\chi) \cap \QQ_p(\theta))].$$
But  $\QQ_{p^a} \cap \QQ_p(\chi) \cap \QQ_p(\theta) \supseteq \QQ_p$ and $[\QQ_{p^a}:\QQ_p]= p^{a-1}$, so in fact
we have shown that
$$\QQ_{p^a}\cap \QQ_p(\theta) \subseteq \QQ_{p^a} \cap \QQ_p(\chi).$$
Denoting
$p^c:= [\QQ_{p^a}:(\QQ_{p^a}\cap \QQ_p(\theta))]$ and $p^d:= [\QQ_{p^a}:(\QQ_{p^a} \cap \QQ_p(\chi))]$,
we now have
\begin{equation}\label{ext21}
  c \geq d.
\end{equation}
By hypothesis, $p^c|\theta(1)$.  Next, $p \nmid |G/N|$ implies that $p \nmid \chi(1)/\chi(1)$, and
so $p^c|\chi(1)$. Hence $p^d|\chi(1)$ by \eqref{ext21}, and we are done.

\smallskip
(ii) It suffices to prove the theorem in the case $G/N$ is cyclic, and furthermore $\theta \in \Irr(N)$ has
$b=\nu_p(c(\theta)) \geq 2$. By Proposition \ref{pro:ext1}(i), for any $\chi \in \Irr(G|\theta)$ we have
$$\nu_p(c(\chi))=b.$$
Since $\mathrm{Stab}_G(\theta)/N$ is cyclic in our case, using Proposition \ref{pro:ext1}(ii)
with $\FF:=\QQ_p$, there exists some $\chi \in \Irr(G|\theta)$ such that
$$p \nmid [\QQ_p(\chi):(\QQ_p(\chi) \cap \QQ_p(\theta))].$$
Arguing as in (i), we see that
$$p \nmid [(\QQ_{p^b}\cap \QQ_p(\chi)):(\QQ_{p^b} \cap \QQ_p(\chi) \cap \QQ_p(\theta))],$$
and hence
$$\QQ_{p^b}\cap \QQ_p(\chi) \subseteq \QQ_{p^b} \cap \QQ_p(\theta).$$
Denoting
$p^c:= [\QQ_{p^b}:(\QQ_{p^b}\cap \QQ_p(\chi))]$ and $p^d:= [\QQ_{p^b}:(\QQ_{p^b} \cap \QQ_p(\theta))]$,
we now have
\begin{equation}\label{ext22}
  c \geq d.
\end{equation}
By hypothesis, $p^c|\chi(1)$.  Next, $p \nmid |G/N|$ implies that $p \nmid \chi(1)/\theta(1)$, and
so $p^c|\theta(1)$. Hence $p^d|\theta(1)$ by \eqref{ext22}, and we are done.
\end{proof}

\begin{thm}\label{main-slu}
The following statements hold.
\begin{enumerate}[\rm(i)]
\item Suppose that $S = \SL_n(q)$ and $p \nmid \gcd(n,q-1)$. Then
Theorem \ref{thm:main-evidence} at the prime $p$ holds for $S = \SL_n(q)$ and for any finite group
$L \rhd S$ with $p \nmid |L/S|$.
\item Suppose that $S = \SU_n(q)$ and $p \nmid \gcd(n,q+1)$. Then
Theorem \ref{thm:main-evidence} at the prime $p$ holds for $S = \SU_n(q)$ and for any finite group
$L \rhd S$ with $p \nmid |L/S|$.
\end{enumerate}
\end{thm}

\begin{proof}
(i)
Embed $S$ in $G:=\GL_n(q)$, and consider
$H:=\ZB(G)S$. Since $G/H \cong \sC_{\gcd(n,q-1)}$, by Theorems \ref{main-gl} and \ref{thm:ext}(ii),
Conjecture \ref{conj:main} holds for $H$ at $p$.
Note that
$$H = \OB_{p'}(\ZB(G))S \times \OB_p(\ZB(G))$$
is a direct product. Indeed, suppose that $\alpha\cdot\mathrm{Id} = \beta g$
where $\alpha\cdot\mathrm{Id} \in \OB_{p}(\ZB(G))$, $\beta\cdot\mathrm{Id} \in \OB_{p'}(\ZB(G))$, and
$g \in S$. So $\alpha \in \FF_q^\times$ is a $p$-element and $\beta \in \FF_q^\times$ is a $p'$-element.
On the other hand,
$$\alpha^n = \det{\alpha\cdot\mathrm{Id}} = \det{\beta g} = \beta^n,$$
so $\alpha^n$ is both a $p$-element and a $p'$-element, whence $\alpha^n=1$.
But $\alpha^{q-1} = 1$, so the order of the $p$-element $\alpha$ divides $\gcd(n,q-1)$.
As $p \nmid \gcd(n,q-1)$, we conclude $\alpha=1$.

As $\OB_{p'}(\ZB(G))S$ is a direct factor of $H$, Conjecture
\ref{conj:main} holds for $\OB_{p'}(\ZB(G))S$. Since $S$ is a
normal subgroup of $\OB_{p'}(\ZB(G))S$ with quotient a cyclic
$p'$-group, so Conjecture \ref{conj:main} holds for $S$ at $p$ by Theorem
\ref{thm:ext}(ii). By Theorem \ref{thm:ext}(i), Conjecture \ref{conj:main} then holds for $L$ at $p$.

\smallskip
(ii) The arguments are entirely similar to the ones in (i).
\end{proof}

\begin{proof}[Proof of Theorem~\ref{thm:main-evidence}]
It follows from
Theorems~\ref{thmD-solvable}, \ref{thmD-alternating},
\ref{main-gl}, and \ref{main-slu}.
\end{proof}


\section{Theorem \ref{thm:main-p'-groups}}\label{sec:4}

For the definition of the \emph{inductive Feit condition}, we refer
the reader to \cite[\S3]{BKNT25}.

\begin{thm}\label{A}
Let $G$ be a finite group and $\chi\in\Irr(G)$ be non-linear. Assume
that simple groups involved in $G$ satisfy the inductive Feit
condition. Then there exist $H<G$ and $\psi\in\Irr(H)$ such that
$\QQ(\chi)\subseteq \QQ(\psi)$ and $[\QQ(\psi):\QQ(\chi)]$ divides
$|G:H|$.
\end{thm}

\begin{proof} Suppose that $N\nor G$, $\theta \in \irr N$ is an irreducible constituent of $\chi_N$
and let $\eta \in \irr{T|\theta}$ be the Clifford correspondent of
$\chi$ over $\theta$. By Lemma~\ref{lem:Cram}, we have that
$[\Q(\eta):\Q(\chi)]$ divides $|G:T|$. If $T<G$, then we are done.
Hence, we may assume that $\chi$ is quasi-primitive. By
\cite[Theorem 4.7]{BKNT25}, there exists $H<G$ such that
$\Q(\chi)=\Q(\psi)$, and we are done. (Notice that this theorem is
stated for primitive characters, but as written after, it also holds
for quasi-primitive characters.)
\end{proof}

The following is Theorem~\ref{thm:main-p'-groups}.

\begin{cor}
Suppose that the finite simple groups satisfy the inductive Feit
condition.
 Then the following hold:
 \begin{enumerate}[\rm(a)]
 \item If $G$ is a finite group and  $\chi \in \irr G$, then $c(\chi)[\Q_{c(\chi)} : \Q(\chi)]$ divides $|G|$.

 \item Let $\FF/\QQ$ be an abelian extension, let $p$ be a prime, and $b\in\ZZ_{\ge 0}$.
 Then $\FF=\Q(\chi)$ for some $\chi \in \irr G$, where $G$ is a finite group with $\nu_p(|G|)=b$, if and only if
\[\nu_p(c(\FF)[\QQ_{c(\FF)}:\FF])\le b.\]
 \end{enumerate}
\end{cor}

\begin{proof}
We argue by induction on $|G|$. If $\chi$ is linear, then
$c(\chi)[\Q_{c(\chi)} : \Q(\chi)]=c(\chi)$ divides the exponent of
$G$, and we are done. We may assume that $\chi$ is not linear. By
Theorem \ref{A}, there exists $H<G$ and $\psi\in\Irr(H)$ such that
$\QQ(\chi)\subseteq \QQ(\psi)$ and $[\QQ(\psi):\QQ(\chi)]$ divides
$|G:H|$. By induction, we have that $c(\psi)[\Q_{c(\psi)} :
\Q(\psi)]$ divides $|H|$. Since $\Q(\chi) \sbs \Q(\psi)$, we have
that $c(\chi)$ divides $c(\psi)$. Let $K=\Q_{c(\chi)} \cap
\Q(\psi)$. By elementary Galois theory, $c(\chi)[\Q_{c(\chi)}:K]$
divides $c(\psi)[\Q_{c(\psi)}:\Q(\psi)]$, which divides $|H|$. Also,
$[K:\Q(\chi)]$ divides $[\Q(\psi):\Q(\chi)]$, which divides $|G:H|$,
and part (a) follows.

Now, let $\FF/\QQ$ be an abelian extension, let $p$ be a prime, and
$b\in\ZZ_{\ge 0}$ such that $\nu_p(c(\FF)[\QQ_{c(\FF)}:\FF])\le b$.
By Lemma~\ref{FG72}, there exists a group $M$ of order
$c(\FF)[\QQ_{c(\FF)}:\FF]$ and $\varphi\in\Irr(M)$ such that
$\QQ(\varphi)=F$. The direct product $G:=M\times C$, where $C$ is
the cyclic $p$-group of order $p^b/(c(\FF)_p[\QQ_{c(\FF)}:\FF]]_p)$,
and the character $\chi:=\varphi\otimes \textbf{1}_C\in\Irr(G)$
satisfy the required condition. Conversely, assume that $G$ is a
finite group such that $|G|_p= p^b$ and $\chi \in \irr G$. By part
(a), we have that $c(\chi)[\Q_{c(\chi)}:\Q(\chi)]$ divides $|G|$,
which implies that $c(\chi)_p[\Q_{c(\chi)}:\Q(\chi)]_p\le
|G|_p=p^b$, and we are done.
\end{proof}

For solvable groups, it was proven in \cite[Theorem~0.1]{Cram88}
that in fact $[\QQ_{c(\chi)}:\QQ(\chi)]$ divides $\chi(1)$, but this
is not true in general (as shown, for instance, by $\Alt_5$ and the irreducible characters of degree 3).


\end{document}